\documentclass[12pt]{amsart}
\usepackage[all]{xy}
\usepackage[utf8]{inputenc}
\usepackage{tikz-cd}
\usepackage{graphicx}
\usepackage{amsthm}
\usepackage{amstext}
\usepackage{amsmath,amscd,amsfonts}
\usepackage{amssymb}
\usepackage{mathrsfs}
\usepackage{mathtools}
\usepackage[all]{xy}
\usepackage{stmaryrd}
\usepackage{color,comment}
\usepackage{amsmath, amsthm, amsfonts, enumerate}

\DeclareMathOperator{\Tor}{Tor}
\DeclareMathOperator{\Ass}{Ass}
\def\N{\mathbb{N}} 
\def\m{\mathfrak{m}}

\def\gr{\text{gr}} \def\Ie{I_{\epsilon}}
\def\grm{\text{gr}_{\mathfrak m}}
\def\im{\text{im }}
\def\ker{\text{ker }}

\newtheorem{theorem}{Theorem}[section]

\newtheorem{corollary}[theorem]{Corollary}
\newtheorem{proposition}[theorem]{Proposition}
\newtheorem{definition}[theorem]{Definition}
\theoremstyle{remark}
\newtheorem{remark}[theorem]{\bf Remark}
\theoremstyle{definition}
\newtheorem{example}[theorem]{\bf Example}
\newtheorem{question}[theorem]{\bf Question}

\numberwithin{equation}{section}

\newcommand{\thistheoremname}{}
\newtheorem{genericthm}[theorem]{\thistheoremname}

  \newtheorem*{genericthm*}{\thistheoremname}
\newenvironment{namedthm*}[1]
  {\renewcommand{\thistheoremname}{#1}%
   \begin{genericthm*}}
  {\end{genericthm*}}  
  
\makeatletter
\@namedef{subjclassname@2020}{%
  \textup{2020} Mathematics Subject Classification}
\makeatother
\begin{document}

\title[Betti numbers under small perturbations]{Betti numbers under small perturbations}

\author[L. Duarte]{Lu\'is Duarte}
\address{Dipartimento di Matematica \\
Universit\`a degli Studi di Genova \\
Via Dodecaneso 35\\
I-16146 Genova, Italia}
\email{duarte@dima.unige.it}

\date{\today}

\begin{abstract}
We study how Betti numbers of ideals in a local ring change under small perturbations. Given $p\in\N$ and given an ideal $I$ of a Noetherian local ring $(R,\m)$, our main result states that there exists $N>0$ such that if $J$ is an ideal with $I\equiv J\bmod \m^N$ and with the same Hilbert function as $I$, then the Betti numbers $\beta_i^R(R/I)$ and $\beta_i^R(R/J)$ coincide for $0\le i\le p$. 
Moreover, we present several cases in which an ideal $J$ such that $I \equiv J \bmod \m^N$ is forced to have the same Hilbert function as $I$, and therefore the same Betti numbers.
\end{abstract}


\maketitle

\section{Introduction}


The goal of this article is to study the behavior of ideals in a Noetherian local ring $(R,\m)$ under small perturbations. Given an ideal $I=(f_1,...,f_r) \subseteq R$ and an integer $N>0$, we will consider 
ideals $J$ of $R$ such that $I\equiv J\bmod \m^N$. This naturally includes 
the study of ideals obtained by perturbing generators of $I$, i.e, ideals of the form $\Ie=(f_1+\epsilon_1,\ldots,f_r+\epsilon_r)$ with $\epsilon_1,\ldots,\epsilon_r\in\m^N$. The latter is particularly related to the finite determinacy problem raised by Samuel. He asked whether, given power series $f_1,\ldots,f_r$ in a formal power series ring $R$, one can always find polynomial truncations $\widetilde{f}_1,\ldots,\widetilde{f}_r$ which define an isomorphic singularity. This question is also particularly relevant from the point of view of the use of computer algebra systems as, while it is possible to implement algorithmic computations with polynomials, in general with series that is not the case. 

In \cite{samuel56algebricite} Samuel showed that, if $f\in S=k\llbracket x_1,...,x_d \rrbracket$ is a hypersurface with an isolated singularity at $\m=(x_1,\ldots,x_d)$, 
and $\epsilon$ is in a sufficiently large power of $\m$, then $S/(f)$ and $S/ (f+\epsilon)$ are isomorphic as $k$-algebras. This result was extended by Hironaka \cite{hironaka1965equivalence}, who showed that given an ideal $I$ of $S$ such that $S/I$ is an equidimensional reduced isolated singularity, there exists $N> 0$ for which $S/I \cong S/J$ whenever $J$ is an ideal of the same height as $I$ such that $I \equiv J \bmod \m^N$, and $S/J$ is equidimensional and reduced. Both Samuel and Hironaka's results were further extended by Cutkosky and Srinivasan \cite{CutkoskySrinivasan}, who include cases of ideals which do not necessarily define isolated singularities, by taking into account the Jacobian ideal of $I$. 

 

Given an ideal $I$ and another ideal $J$, with $J \equiv I \bmod \m^N$, even when $R/I$ and $R/J$ are not isomorphic, one can still try to compare their singularities. Several efforts have been made in the literature to find conditions that ensure that $R/I$ and $R/J$ share some common features; for instance, the Hilbert function, among others. 
In this direction, Srinivas and Trivedi show in \cite{srinivas1996invariance} that, if $(R,\m,k)$ is a generalized Cohen-Macaulay local ring and $f_1,\ldots,f_r$ is part of a system of parameters of $R$, then the Hilbert function of $I=(f_1,\ldots,f_r)$ coincides with that of $\Ie=(f_1+\epsilon_1,\ldots,f_r+\epsilon_r)$ for any $\epsilon_1,\ldots,\epsilon_r\in \m^N$, with $N$ sufficiently large. This was greatly extended by Ma, Pham Hung Quy and Smirnov \cite{ma2019filter}, who proved that
the same conclusion holds true in any local ring $R$ if $f_1,\ldots,f_r$ is a filter-regular sequence.  Furthermore, in \cite{quy2021when}, an explicit upper bound for the smallest $N$ with the above property was obtained.

The main goal of this article is to study whether it is possible to preserve another important class of numerical invariants, the Betti numbers of the ideal $I$ as an $R$-module, under small perturbations. Our main result positively answers this question under the assumption that $I$ and its perturbation $J$ have the same Hilbert function (observe that in general two ideals with the same Hilbert function do not have the same Betti numbers, not even the same projective dimension).  We also provide examples that justify the need for this assumption (see Examples \ref{ExampleConjecture} and \ref{Example nCM}).

\begin{namedthm*}{Main result}
Let $I$ be an ideal of a Noetherian local ring $(R,\m)$ and let $p\in\N$. There exists $N\in\N$, depending only on $I$ and $p$, such that for every ideal $J$ satisfying $J \equiv I \bmod \m^N$, one has 
$\beta_i^R(R/I)=\beta_i^R(R/J)$ for all $0\le i\le p$,
provided $I$ and $J$ have the same Hilbert function.
\end{namedthm*}

Actually, we prove a stronger statement saying that the first $p$ maps in a minimal free resolution of $R/J$ can be obtained as a perturbation of the corresponding $p$ maps in a minimal free resolution of $R/I$ (see Theorem \ref{T-perturbatingcomplex}). Then, in Corollary \ref{T-SameBettiNumbers}, we deduce the above main result.
Furthermore, we provide a specific value for the natural number $N$ for which the statement holds. This value for $N$ will depend on the Artin-Rees numbers with respect to $\m$ of the first $p+1$ syzygy modules of $R/I$.
The main tool which is used in the proof comes from a result due to Eisenbud \cite{eisenbud1974adic}, which relates the homology of a given complex $C$ and that of a complex $C_\epsilon$ which is obtained by perturbing the maps of $C$ by maps $\epsilon$ with image in a sufficiently large power of $\m$ (see Theorem \ref{T-EisenbudComplexPerturbation}). 

A related result was obtained by Eisenbud and Huneke, in \cite{EisenbudHuneke}, where they prove that if $R$ is Cohen-Macaulay, and a given complex is obtained by perturbing the maps of a free resolution of $R/I$, then such a complex is itself a resolution. We point out a subtle but crucial difference with our main theorem: given any perturbation $J$ of $I$ with the same Hilbert function, we build a complex (in fact, a resolution) of $R/J$ which is a perturbation of the resolution of $R/I$. The result of Eisenbud and Huneke for Cohen-Macaulay rings, on the other hand, while it has no assumption on Hilbert functions, requires that a given perturbation of the minimal resolution of $R/I$ is already a complex, which then turns out to be a resolution.

The article is structured as follows: Section \ref{preliminaries} includes some preliminary results concerning the theory of standard bases as well as details on Eisenbud's result on perturbations of complexes. 
Section \ref{SectionBetti} contains the main results on the behavior of Betti numbers under small perturbations. 
Finally, Section \ref{SectionHilbert} presents several applications of the main results, together with meaningful examples. Moreover, we identify new cases in which a given ideal is forced to have the same Hilbert function as its perturbations, yielding cases in which our main theorem can be applied. Our analysis includes an extension of a result of Elias (see Theorem \ref{P-HighDepthPertutbation}). 
\newpage
\section{Preliminaries}\label{preliminaries}

We start by recalling notions, terminology and some preliminary results which will be used throughout this article.

$(R,\m)$ will denote a Noetherian local ring and $k=R/\m$ will denote its residue field. We will not require any assumptions on $k$. An element $f \in R$ is called \emph{filter-regular} if $\Ass_R((0:(f))) \subseteq \{\m\}$. A sequence of elements $f_1,\ldots,f_r$ of $R$ is called a \emph{filter-regular sequence} if, for every $1 \leq i \leq r$, the image of $f_i$ in $R/(f_1,\ldots,f_{i-1})$ is a filter-regular element. Given a finitely generated $R$-module $M$, we let $\beta_i^R(M) = \dim_k(\Tor_i^R(M,k))$ be its $i$-th Betti number. 

The \emph{Hilbert function} of $R$ is defined to be the Hilbert function of its associated graded ring $\gr_\m(R) = \bigoplus_{i=0}^\infty \m^i/\m^{i+1}$, which is a standard graded $k$-algebra. 


An \emph{$\m$-filtration} $\mathbb M$ of $M$ is a collection $\{M_i\}_{i\ge 0}$ of submodules of $M$ such that $\m M_i\subseteq M_{i+1}\subseteq M_{i}$ for every $i\ge 0$. 
The filtration $\mathbb M$ is called \emph{stable} (or good) if $\m M_n=M_{n+1}$ for all sufficiently large $n$. Given an $\m$-filtration $\mathbb M$, we define the 
\emph{associated graded module} of $M$ with respect to $\mathbb M$ as 
$$\gr_{\mathbb M}(M)=\bigoplus_{i=0}^{\infty}{\frac{M_i}{M_{i+1}}}.$$

Since $\mathbb M$ is an $\m$-filtration, $\gr_{\mathbb M}(M)$ has a natural structure of a graded module over $\grm(R)$.
In case $\mathbb M=\{\m^iM\}_{i\ge 0}$ is the $\m$-adic filtration of $M$, we will denote $\gr_{\mathbb M}(M)$ simply by $\grm(M)$. 



Given a nonzero element $x\in M$ we define $v(x)$ as the unique integer $i$ such that $x\in \m^iM$ and $x\not\in\m^{i+1}M$. We denote by $x^*\in\grm(M)$ the \emph{initial form} of $x$, that is, the image of $x$ inside $\m^{v(x)}M/\m^{v(x)+1}M$. 
If $N$ is a submodule of 
$M$, we define the \emph{initial module} of $N$, denoted by $N^*$, as the kernel of the graded map of $\grm(R)$-modules 
$$\grm(M)\to \grm(M/N)$$ 
induced by the projection $M\to M/N$. Thus, $\grm(M/N)\cong \grm(M)/N^*$.  Whenever we write $N^*$, the inclusion $N\subseteq M$ will be implicit from the context.

\begin{remark}
Observe that $N^*$ is the $\grm(R)$-submodule of $\grm(M)$ generated by the set $\{x^*\in\grm(M) \mid x\in N\}$. In fact, we have
$$N^*=\bigoplus_{i=0}^{\infty}{\frac{N\cap \m^iM+\m^{i+1}M}{\m^{i+1}M}} \cong \bigoplus_{i=0}^{\infty}{\frac{N\cap \m^iM}{N\cap \m^{i+1}M}}.$$
Therefore $N^*$ can also be viewed as the associated graded module of $N$ with respect to the filtration $\{N\cap \m^iM\}_{i\ge 0}$, which is stable by the Artin-Rees Lemma. 
\end{remark}

\begin{remark}
In the case of an ideal $I \subseteq R$, the ideal $I^*$ of $\gr_\m(R)$ is called the initial ideal of $I$. In what follows, we will refer to the Hilbert function of an ideal $I \subseteq R$, denoted by $\text{HF}_{R/I}$, to mean the Hilbert function of the local ring $R/I$, that is, the Hilbert function of $\gr_{\m/I}(R/I)$. Hence, $\text{HF}_{R/I}(n)=\text{HF}_{\grm(R)/I^*}(n)$ for all $n\ge 0$.
\end{remark}

A subset $\{x_1,\ldots,x_r\}$ of $N$ is said to be a \emph{standard basis} of $N$ if $\{x_1^*,\ldots,x_r^*\}\subseteq \grm(M)$ generates the $\grm(R)$-module $N^*$. While not all generating sets of $N$ are a standard basis, it is true that every standard basis of $N$ is a generating set (see \cite[Proposition 2.1]{herzog2016homology}). A standard basis of $N$ is said to be \emph{minimal} if no strict subset of it is a standard basis of $N$. For more details on these concepts we refer the reader to \cite{rossi2010consecutive}, \cite{rossi2009minimal} and \cite{shibuta2008cohen}. It will be important for us to consider minimal standard bases since they are related to Artin-Rees numbers, whose definition we recall next.

\begin{definition}
Let $R$ be a Noetherian ring, $I$ an ideal of $R$, and $N\subseteq M$ be finitely generated $R$-modules. The Artin-Rees number $\emph{\text{AR}}(I,N\subseteq M)$ is the least integer $s$ such that
$$I^nM\cap N=I^{n-s}(I^sM\cap N)\quad \text{for every } n\ge s.$$
\end{definition}

By \cite[Proposition 2.1 (a)]{herzog2016homology}, if $\{x_1,\ldots,x_r\}$ is a minimal standard basis of $N$, then $\text{AR}(\m,N \subseteq M) = \max\{v(x_i),\ldots, v(x_r)\}$. In particular, if $x_1,\ldots,x_r$ are elements of $N$ such that $x_1^*,\ldots,x_r^*$ generate $N^*$ as a $\gr_\m(R)$-module up to degree $\text{AR}(\m,N \subseteq M)$, then they must generate $N^*$.


\subsection{Approximations of complexes:} We conclude this preliminary section by discussing a result 
on how the homology of a complex changes under perturbations. Let
$$\begin{tikzcd}
C: \cdots \arrow[r] & C_{n+1} \arrow[r, "f_{n+1}"] & C_{n} \arrow[r, "f_n"] & C_{n-1} \arrow[r] & \cdots \end{tikzcd}$$ 
be a complex of finitely generated $R$-modules.
An $\m$\emph{-adic approximation of} $C $ of order $d=(\ldots,d_{n+1},d_n,d_{n-1},\ldots)\in\mathbb N^{\mathbb Z}$ is a complex $C _{\epsilon}$ of the form $$\begin{tikzcd}[row sep=large,column sep = large]
C _{\epsilon}: \cdots \arrow[r] & C_{n+1} \arrow[r, "f_{n+1}+\epsilon_{n+1}"] & C_{n} \arrow[r, "f_n+\epsilon_n"] & C_{n-1} \arrow[r] & \cdots\end{tikzcd}$$
where $\epsilon_n$ is a map from $C_n$ to $\m^{d_n}C_{n-1}$ for all $n$.

For any integer $n$, let $H_n(C)$ denote the $n$-th homology module of $C$. We will consider the initial module $H_n(C )^* = \bigoplus_p H_n(C)^*_p$ of $H_n(C)$ with respect to the inclusion $H_n(C )\subseteq C_n/\text{im }(f_{n+1})$, where the latter has a natural $\m$-adic filtration induced by $\{\m^i C_n\}_{i \ge 0}$, and we want to compare such an initial module with the one obtained from a perturbed complex $C_\epsilon$. We warn the reader that the initial module of $H_n(C)$ is computed inside $\gr_\m(C_n/{\rm im}(f_{n+1}))$, while the initial module of $H_n(C_\epsilon)$ is computed inside $\gr_\m(C_n/{\rm im}(f_{n+1}+\epsilon_{n+1}))$. 
Despite this difference, Eisenbud's main result in  \cite{eisenbud1974adic} shows a close relation between them.
The following technical result highlights information that can essentially be deduced from Eisenbud's proof.

\begin{theorem}\label{T-EisenbudComplexPerturbation}
Let $C $ be a complex of finitely generated $R$-modules. There exists a sequence of integers $d=(\ldots,d_{n+1},d_n,d_{n-1},\ldots)$ such that, if $C _{\epsilon}$ is an $\m$-adic approximation of $C$ of order $d$, then \begin{enumerate}[(i)]
\item $H_n(C_\epsilon )^*$ is a subquotient of $H_n(C )^*$ for all $n$.
\item If $H_n(C )$ and $H_{n-1}(C )$ are both annihilated by some power of $\m$, then $H_n(C )^*\cong H_n(C_\epsilon )^*$. 
\end{enumerate}
Moreover, letting $s_n=\emph{\text{AR}}(\m,\emph{\text{im}}(f_{n+1})\subseteq C_n)$, for (i) to hold it suffices to take $$d_n=1+\max\{s_n,s_{n-1},s_{n-2}\}$$ for every $n$. For (ii) to hold it suffices to take $$d_{n+1}=d_n=d_{n-1}=1+\max\{s_n,s_{n-1},s_{n-2},q\},$$ where $q$ is such that $\m^q$ annihilates $H_n(C )$ and $H_{n-1}(C )$.
\end{theorem}

\begin{remark}
From Eisenbud's proof of Theorem \ref{T-EisenbudComplexPerturbation} one can also deduce that, for every integer $n$, the $k$-vector space $H_n(C_\epsilon )^*_p$ is a subquotient of $H_n(C )^*_p$ for every integer $p$. In other words, $H_n(C_\epsilon)^* = \bigoplus_p H_n(C_\epsilon)^*_p$ is a \emph{graded} subquotient of $H_n(C)^* = \bigoplus_p H_n(C)^*_p$. For completeness, in order to justify both this and the stated bound for the $d_n$, we include a sketch of Eisenbud's proof.
\end{remark} 
\begin{proof}
The proof is done by comparing the spectral sequences associated to both complexes $C$ and $C_{\epsilon}$. 
By setting 
$$
^{r}Z_n^p=\{x\in\m^pC_n:f_n(x)\in\m^{p+r}C_{n-1}\}$$$$
^{r}B_n^p=\m^pC_n\cap f_{n+1}(\m^{p-r}C_{n+1})
$$ for every $r\ge0$ and every $n,p\in\mathbb Z$, 
the usual spectral sequence $\{^{r}E_n^p\}_{n,p\in\mathbb Z}$ associated to the complex $C$ is given by 
$$^{r}E_n^p=\frac{^{r}Z_n^p+\m^{p+1}C_n}{^{r-1}B_n^p+\m^{p+1}C_n}\cong 
\frac{^{r}Z_n^p}{^{r-1}B_n^p+{^{r-1}Z_n^{p+1}}}.$$
This spectral sequence converges to the homology of $C$.

\vspace{0.2cm}
\emph{Claim:} We claim that for fixed $n$ we have $^{r}E_n^p={^{r(n)}E_n^p}$ for every $p$ and every $r\ge r(n)$, where $r(n)=1+\max\{s_n,s_{n-1}\}$. 
\begin{proof}
 We start by observing that $^{r-1}B_n^p={^{r(n)-1}B_n^p}$ for every $r\ge r(n)$. Indeed, the inclusion $\supseteq$ holds since $\m^{p-r+1}C_n\supseteq \m^{p-r(n)+1}C_n$, while the inclusion $\subseteq$ holds since $^{r-1}B_n^p \subseteq \m^p C_n$ and also
\begin{equation*}
\begin{split}
^{r-1}B_n^p & = \m^pC_n\cap f_{n+1}(\m^{p-r+1}C_{n+1}) \subseteq \m^pC_n\cap \im(f_{n+1})
\\ & \subseteq \m^{p-s_n}\im(f_{n+1})
\subseteq f_{n+1}(\m^{p-r(n)+1}C_{n+1}).
\end{split}
\end{equation*}
Now we prove that $^{r}Z_n^p+\m^{p+1}C_n={^{r(n)}Z_n^p+\m^{p+1}C_n}$ for every $r\ge r(n)$. The inclusion $\subseteq$ is trivial since $\m^{p+r}C_{n-1}\subseteq \m^{p+r(n)}C_{n-1}$. To show $\supseteq$ we observe that, given $x\in {^{r(n)}Z_n^p}$,
$$f_n(x)\in \m^{p+r(n)}C_{n-1}\cap \im(f_n)\subseteq \m^{p+r(n)-s_{n-1}}\im(f_n)\subseteq \m^{p+1}\im(f_n),$$ thus showing that $x\in \ker(f_n)+\m^{p+1}C_n$. It follows that $x\in \m^pC_n\cap(\ker(f_n)+\m^{p+1}C_n)=(\m^pC_n\cap\ker(f_n))+\m^{p+1}C_n\subseteq {^{r}Z_n^p+\m^{p+1}C_n}$.
\end{proof}

Consequently, $H_n(C)_p^*={^{r(n)}E_n^p}$ for all $n,p$.
Let $\{^{r}F_n^p\}_{n,p\in\mathbb Z}$ be the spectral sequence associated to $C_{\epsilon}$. In \cite{eisenbud1974adic}, Eisenbud shows that by choosing $d_n=\max\{r(n),r(n-1)\}=1+\max\{s_n,s_{n-1},s_{n-2}\}$ for all $n$, we have $^{r(n)}E_n^p={^{r(n)}F_n^p}$ for all $n,p$. Since $^{\infty}F_n^p$ is a subquotient of $^{r}F_n^p$ for every $r$, $n$ and $p$, it follows that 
$H_n(C_\epsilon)^*=\oplus_{p}{^{\infty}F_n^p}$ is a graded subquotient of $\oplus_{p}{^{r(n)}F_n^p}\cong\oplus_{p}{^{r(n)}E_n^p}\cong H_n(C )^*$.

The bound for $d_{n+1},d_n$ and $d_{n-1}$ on (ii) also follows from Eisenbud's proof and our expression for $r(n)$.
\end{proof}

An immediate corollary of Theorem \ref{T-EisenbudComplexPerturbation}, which will be crucial in the proof of our main theorem, is the following.

\begin{corollary}\label{CorollaryEisenbud}
Let $f:M\to N$ be an $R$-homomorphism of finitely generated modules over the local ring $(R,\m)$ and let $d=1+\emph{\text{AR}}(\m,\emph{\im}(f)\subseteq N)$. Then for any map $\epsilon:M\to \m^dN$ we have that $(\ker(f+\epsilon))^*$ is a graded subquotient of $(\ker (f))^*$. Here, the initial modules of $\ker(f)$ and $\ker(f+\epsilon)$ are both computed inside $\emph{\text{gr}}_{\m}(M)$.
\end{corollary}
\begin{proof}
Apply Theorem \ref{T-EisenbudComplexPerturbation} to the complex
$\begin{tikzcd}
0 \arrow[r] & M \arrow[r, "f"] & N \arrow[r] & 0
\end{tikzcd}.$ 
\end{proof}

\section{Betti numbers under small perturbations}\label{SectionBetti}

We begin this section with an observation regarding the inital module of a of a pertubed submodule.

\begin{proposition}\label{P-starcontained}
Let $(R,\m)$ be a Noetherian local ring, $M$ a finitely generated $R$-module and $L$ a submodule of $M$. For $N>\emph{\text{AR}}(\m,L\subseteq M)$, if $K$ is a submodule of $M$ such that $L\equiv K \bmod \m^NM$, then $L^*\subseteq K^*$.
\end{proposition}
\begin{proof}
Consider a minimal standard basis $\{x_1,\ldots,x_m\}$ of $L$. As $L\equiv K \bmod \m^NM$, there are $\epsilon_1,\ldots,\epsilon_m\in\m^NM$ such that $x_i+\epsilon_i\in K$ for $i=1,\ldots,m.$ By \cite[Proposition 2.1]{herzog2016homology}, $\text{AR}(\m,L\subseteq M)=\max\{v(x_1),\ldots,v(x_m)\}$, hence the choice $N>\text{AR}(\m,L\subseteq M)$ implies that $(x_i+\epsilon_i)^*=x_i^*$ for every $i=1,\ldots,m$. We conclude that \par
\vspace{\abovedisplayshortskip}
\hfill
$L^*=(x_1^*,\ldots,x_m^*)=((x_1+\epsilon_1)^*,\ldots,(x_m+\epsilon_m)^*)\subseteq K^*.$ \end{proof}

This has the following implication.

\begin{proposition}\label{Cor-EqualStar}
Let $(R,\m)$ be a Noetherian local ring and $I$ be an ideal of $R$. For $N>\emph{\text{AR}}(\m,I\subseteq R)$, if $J$ is an ideal of $R$ such that $I\equiv J \bmod \m^N$, then 
\begin{enumerate}[(i)]
\item $\emph{\text{HF}}_{R/I}(n)\ge \emph{\text{HF}}_{R/J}(n)$ for every $n\in\N$. In particular $\dim(R/I)\ge \dim(R/J)$.
\item $\emph{\text{HF}}_{R/I}(n)= \emph{\text{HF}}_{R/J}(n)$ for every $n\in\N$ if and only if $\emph{\text{\gr}}_{\m}(R/I)\cong\emph{\text{\gr}}_{\m}(R/J)$ if and only if $I^*=J^*$.
\end{enumerate}
\end{proposition}
\begin{proof}
We have $\grm(R/I)=\grm(R)/I^*$ and $\grm(R/J)=\grm(R)/J^*$. Since $I^*\subseteq J^*$ by Proposition \ref{P-starcontained}, it follows that $\text{HF}_{R/I}(n)\ge \text{HF}_{R/J}(n)$ for all $n\ge 0$, with equality if and only if $I^*=J^*$. Also, $I^*=J^*$ implies $\grm(R/I)\cong \grm(R/J)$, which in turn implies $\text{HF}_{R/I}= \text{HF}_{R/J}$.
\end{proof}
We remark that point (i) in Proposition \ref{Cor-EqualStar} has been proved by Srinivas and Trivedi in \cite[Lemma 3]{srinivas1996invariance}, although our proof is different. 

If $R$ is Cohen-Macaulay, or even just equidimensional, then (i) says that the height cannot decrease after perturbation. Still, it is not true that in general the height is preserved under any sufficiently small perturbations. The following example, privately communicated to us by Pham Hung Quy and Rossi, shows that there are ideals for which the height is not preserved under truncations on the generators of an ideal.   

\begin{example}
Consider $R=k\llbracket x,y \rrbracket$ and $I=(f_1,f_2)$, where
$$\begin{array}{c}
f_1=(x+y+y^2)(xy+y^2+y^3+\cdots) \\ 
f_2=(x+y+y^2)(xy+x^2+x^3+\cdots).
\end{array}$$ 
$I$ has height 1 because $I\subseteq (x+y+y^2)$. For $N,N' \geq 3$, we consider the truncations of $f_1$ at degree $N$ and of $f_2$ at degree $N'$: 
$$\begin{array}{c}
(f_1)_N=(x+y+y^2)(xy+y^2+y^3+\cdots+y^{N-2})+y^N+xy^{N-1} \\ 
(f_2)_{N'}=(x+y+y^2)(xy+x^2+x^3+\cdots+x^{N'-2})+x^{N'}+x^{N'-1}y.
\end{array}$$ 
It can be checked that $I_{N,N'}=((f_1)_N,(f_2)_{N'})$ has height $2$, and therefore $\dim(R/I) = 1 > \dim(R/I_N) = 0$. Thus, the inequality in Proposition \ref{Cor-EqualStar} (i) can be strict. This example is particularly interesting in the direction of the finite determinacy problem, since it shows that no polynomial truncations of $f_1$ and $f_2$ generate an ideal of the same height as $I$. 
\end{example}

We proceed towards showing some first results concerning the the behavior of Betti numbers under small perturbations. The following proposition shows that there is an inequality for the zero-th Betti number of an ideal, i.e., the minimal number of generators, which we denote by $\mu(-)$.

\begin{proposition}\label{P-NumberGeneratorsInequality}
Let $(R,\m)$ be a Noetherian local ring and let $I$ be an ideal of $R$. For $N>\emph{\text{AR}}(\m,I\subseteq R)$, if $J$ is an ideal of $R$ such that $I\equiv J \bmod \m^N$, then $\mu(I)\le\mu(J)$.
\end{proposition}
\begin{proof}
Our choice of $N$ guarantees that $\m^N\cap I\subseteq \m I$. 
Consider $\{f_1,\ldots,f_m\}$ to be a minimal system of generators of $I$. As $I\equiv J \bmod \m^N$, there are $\epsilon_1,\ldots,\epsilon_m\in\m^N$ such that $f_i+\epsilon_i\in J$ for $i=1,\ldots,m.$ To finish the proof it is enough to show that $\{f_1+\epsilon_1,\ldots,f_m+\epsilon_m\}$ is part of a minimal system of generators of $J$. By Nakayama, this is equivalent to proving that the image of $\{f_1+\epsilon_1,\ldots,f_m+\epsilon_m\}$ in $J/\m J$ is an $R/\m$-linearly independent set; so supposing $a_1,\ldots,a_m\in R$ are such that $\sum_{i=1}^{m}{a_i(f_i+\epsilon_i)}\in\m J$, we need to show that $a_1,..,a_m\in\m$. Indeed, this implies that
\begin{equation*}
\begin{split}
\sum_{i=1}^{m}{a_if_i}\in I\cap\left(\m J+ \left(\sum_{i=1}^{m}{a_i\epsilon_i}\right)\right)& \subseteq I\cap(\m J+\m^N)=I\cap(\m(J+\m^{N-1}))
\\ & = I\cap(\m(I+\m^{N-1}))
\\ & = \m I+\m^N\cap I\subseteq \m I.
\end{split}
\end{equation*}
Because $\{f_1,\ldots,f_m\}$ is a minimal system of generators of $I$, we conclude that $a_1,..,a_m\in\m$. 
\end{proof}

Example \ref{Ex-BetaOne} ahead shows that in general the inequality in Proposition \ref{P-NumberGeneratorsInequality} can be strict. This means that, if we want to obtain preserverance of the Betti numbers unders small perturbations, we must require stronger hypothesis on the perturbed ideal. 
As it turns out, assuming that the Hilbert function is preserved will be enough to obtain equality of Betti numbers. 
The following proposition is the first step toward proving the main result.

\begin{proposition}\label{P-NumberGeneratorsEquality}
Let $(R,\m)$ be a Noetherian local ring, $I$ be an ideal of $R$ and $N$ be an integer such that $N>\emph{\text{AR}}(\m,I\subseteq R)$. Let $J$ be an ideal of $R$ such that $I\equiv J \bmod \m^N$, and $I$ and $J$ have the same Hilbert function. Then $\mu(I)=\mu(J)$. Moreover, there exist a minimal set of generators $\{f_1,\ldots,f_m\}$ of $I$ and $\epsilon_1,\ldots,\epsilon_m\in\m^N$ such that $J=(f_1+\epsilon_1,\ldots,f_m+\epsilon_m).$
\end{proposition}
\begin{proof}
By Proposition \ref{Cor-EqualStar}, we have $I^*=J^*$. Consider a minimal standard basis $\{f_1,\ldots,f_m,\\ g_1,\ldots,g_r\}$ of $I$ such that $\{f_1,\ldots,f_m\}$ is a minimal set of generators of $I$. By the proof of Proposition \ref{P-NumberGeneratorsInequality} there are $\epsilon_1,\ldots,\epsilon_m\in\m^N$ for which $f_1+\epsilon_1,\ldots,f_m+\epsilon_m$ are part of a minimal set of generators of $J$. The proof will be concluded if we show that the $f_1+\epsilon_1,\ldots,f_m+\epsilon_m$ generate $J$. For each $j=1,\ldots,r$, write $g_j=\sum_{k=1}^{m}{a_{jk}f_k}$ and let $\delta_j=\sum_{k=1}^{m}{a_{jk}\epsilon_k}\in\m^N$, so that $g_j+\delta_j=\sum_{k=1}^{m}{a_{jk}(f_k+\epsilon_k)}\in J$. Since, by \cite[Proposition 2.1]{herzog2016homology}, $N>\text{AR}(\m,I\subseteq R)=\max\{v(f_1),\ldots,v(f_m), v(g_1),\ldots,v(g_r)\}$, we have
$$J^*=I^*=(f_1^*,\ldots,f_m^*,g_1^*,\ldots,g_r^*)=((f_1+\epsilon_1)^*,\ldots,(f_m+\epsilon_m)^*,(g_1+\delta_1)^*,\ldots,(g_r+\delta_r)^*),$$ 
thus showing that $\{f_1+\epsilon_1,\ldots,f_m+\epsilon_m,g_1+\delta_1,\ldots,g_r+\delta_r\}$ is a standard basis of $J$. As a consequence, it is also a generating set of $J$, meaning that \par
\vspace{\abovedisplayshortskip}
\hfill
$J=(f_1+\epsilon_1,\ldots,f_m+\epsilon_m,g_1+\delta_1,\ldots,g_r+\delta_r)=(f_1+\epsilon_1,\ldots,f_m+\epsilon_m).$ 
\end{proof}

We now come to our key result.

\begin{theorem}\label{T-perturbatingcomplex}
Let $I$ be an ideal of a Noetherian local ring $(R,\m)$ and let 
$$\begin{tikzcd}
\cdots \arrow[r] & F_2 \arrow[r, "d_2"] & F_1 \arrow[r, "d_1"] & F_0 \arrow[r] & R/I \arrow[r] & 0
\end{tikzcd}$$
be a minimal free resolution of $R/I$. Set $s_n=\emph{\text{AR}}(\m,\emph{\text{im}}(f_{n+1})\subseteq F_n)$ for all $n\ge 0$ and let $p$ and $N_0$ be natural numbers such that $N_0>\max\{s_i:i=0,...,p\}$.

For $N>N_0+\sum_{i=0}^{p-1}{s_i}$ the following holds.
Given an ideal $J$ such that $J \equiv I \bmod \m^N$ and such that $I$ and $J$ have the same Hilbert function, there exist maps $\delta_i:F_i\to \m^{N_0}F_{i-1}$, $i=1,...,p+1$, for which
$$\begin{tikzcd}[row sep=large,column sep = large]
F_{p+1} \arrow[r, "d_{p+1}+\delta_{p+1}"] & F_p \arrow[r, "d_p+\delta_p"] & \cdots \arrow[r, "d_2+\delta_2"] & F_1 \arrow[r, "d_1+\delta_1"] &  F_0 \arrow[r] & R/J \arrow[r] & 0,
\end{tikzcd}$$
is exact; moreover, $(\emph{\im}d_i)^*=(\emph{\im}(d_i+\delta_i))^*$ for all $i=1,...,p+1$ (here both initial modules are computed inside $\emph{\text{gr}}_{\m}(F_{i-1})$).
\end{theorem}
\begin{proof}
We prove the result by induction on $p$.

Suppose first $p=0$. The map $d_1:F_1\to F_0=R$ is given by a matrix $[f_1\cdots f_m]$, where $\{f_1,\ldots,f_m\}$ is a minimal generating set of $I$. Since $N>s_0=\text{AR}(\m,I\subseteq R)$, by Proposition \ref{P-NumberGeneratorsEquality} there exist $\epsilon_1,\ldots,\epsilon_m\in\m^N$ such that $\{f_1+\epsilon_1,\ldots,f_m+\epsilon_m\}$ is a minimal generating set of $J$. Thus it is enough to let $\delta_1:F_1\to R=F_0$ be given by the matrix $[\epsilon_1\cdots \epsilon_m]$. Indeed, as $I$ and $J$ have the same Hilbert function, by Proposition \ref{Cor-EqualStar} it follows that $(\im (d_1+\delta_1))^*=J^*=I^*=(\im d_1)^*$.

Suppose now that $p\ge 1$ and that the result is valid for $p-1$. We are assuming that $N>N_0+\sum_{i=0}^{p-1}s_i$, so, by induction hypothesis, for each $i=1,...,p$ there are maps $\delta_i:F_i\to\m^{N_0+s_{p-1}}F_{i-1}$ such that 
$$\begin{tikzcd}[row sep=large,column sep = large]
F_{p} \arrow[r, "d_{p}+\delta_{p}"] & \cdots \arrow[r, "d_2+\delta_2"] & F_1 \arrow[r, "d_1+\delta_1"] & F_0 \arrow[r] & R/J \arrow[r] & 0
\end{tikzcd}$$
is an exact complex. Also, $(\im d_p)^*=(\im(d_p+\delta_p))^*$. 
By \cite[Proposition 2.1]{herzog2016homology} this implies that $\text{AR}(\m,\im (d_{p}+\delta_{p})\subseteq F_{p-1})=\text{AR}(\m,\im d_{p}\subseteq F_{p-1})=s_{p-1}$. 
For each element $e$ of the canonical basis of $F_{p+1}$ we observe that 
$$(d_{p}+\delta_{p})(d_{p+1}(e))=\delta_{p}(d_{p+1}(e))\in \m^{N_0+s_{p-1}}F_{p-1}\cap \im (d_{p}+\delta_{p})\subseteq \m^{N_{0}}(\im (d_{p}+\delta_{p})),$$ 
hence we can find $e'\in \m^{N_0}F_{p}$ for which $(d_{p}+\delta_{p})(d_{p+1}(e))=(d_{p}+\delta_{p})(e')$. We define $\delta_{p+1}:F_{p+1}\to \m^{N_0}F_{p}$ by setting $\delta_{p+1}(e)=-e'$ for every element $e$ of the canonical basis of $F_{p+1}$. 
We thus have that
$$\begin{tikzcd}[row sep=large,column sep = large]
F_{\delta}: F_{p+1} \arrow[r, "d_{p+1}+\delta_{p+1}"] & F_{p} \arrow[r, "d_{p}+\delta_{p}"] & \cdots \arrow[r, "d_2+\delta_2"] & F_1 \arrow[r, "d_1+\delta_1"] & F_0 
\end{tikzcd}$$ is a complex. 
Next we show that $(\im (d_{p+1}+\delta_{p+1}))^*=(\im d_{p+1})^*$.
Since $N_0>s_{p}$ and $\im (d_{p+1}+\delta_{p+1})\equiv \im d_{p+1} \bmod \m^{N_0}F_p$, by Proposition \ref{P-starcontained} it follows that
$$(\ker d_{p})^*=(\im d_{p+1})^*\subseteq (\im (d_{p+1}+\delta_{p+1}))^*\subseteq (\ker (d_{p}+\delta_{p}))^*.
$$
Since also $N_0>s_{p-1}$, by Corollary \ref{CorollaryEisenbud} we have that $(\ker (d_{p}+\delta_{p}))^*$ is a graded subquotient of  $(\ker d_{p})^*$. This together with the inclusion $(\ker d_{p})^*\subseteq (\ker (d_{p}+\delta_{p}))^*$ implies the equality $(\ker d_{p})^*= (\ker (d_{p}+\delta_{p}))^*$. Consequently, $(\im (d_{p+1}+\delta_{p+1}))^*=(\im d_{p+1})^*$.

Finnaly, we show that complex $F_{\delta}$ is exact at $F_p$. 
$F_{\delta}$ is an $\m$-adic aproximation of 
$$\begin{tikzcd}[row sep=large,column sep = large]
F: F_{p+1} \arrow[r, "d_{p+1}"] & F_p \arrow[r, "d_p"] & \cdots \arrow[r, "d_2"] & F_1 \arrow[r, "d_1"] & F_0
\end{tikzcd} $$
of order $(\ldots,N_0,N_0,N_0,\ldots)$. Since $N_0>\max\{s_i:i=0,\ldots,p\}$, by Theorem \ref{T-EisenbudComplexPerturbation} we can conclude that $H_p(F_{\delta})^*$ is a subquotient of $H_p(F)^*=0$. Hence $H_p(F_{\delta})=0$, as we wanted to show.
\end{proof}

\begin{corollary}\label{T-SameBettiNumbers}
Let $I$ be an ideal of a Noetherian local ring $(R,\m)$ and let $p\in\N$. There exists $N\in\N$ with the following property. For every ideal $J$ such that $J \equiv I \bmod \m^N$, one has $\beta_i^R(R/I)=\beta_i^R(R/J)$ for all $0\le i\le p$, provided $I$ and $J$ have the same Hilbert function.

Consequently, if $N$ is large enough, for such an ideal $J$ we have that $R/I$ and $R/J$ have the same projective dimension. 
\end{corollary}
\begin{proof}
Fix a minimal free resolution 
$$\begin{tikzcd}[row sep=large,column sep = large]
\cdots \arrow[r] & F_{p+1} \arrow[r, "d_{p+1}"] & F_p \arrow[r, "d_p"] & \cdots \arrow[r, "d_2"] & F_1 \arrow[r, "d_1"] & F_0 
\end{tikzcd} $$
of $R/I$. By Theorem \ref{T-perturbatingcomplex}, given $N\gg 0$, there is an exact complex 
$$\begin{tikzcd}[row sep=large,column sep = large]
F_{\delta}: F_{p+1} \arrow[r, "d_{p+1}+\delta_{p+1}"] & F_p \arrow[r, "d_p+\delta_p"] & \cdots \arrow[r, "d_2+\delta_2"] & F_1 \arrow[r, "d_1+\delta_1"] &  F_0 \arrow[r] & R/J \arrow[r] & 0,
\end{tikzcd}$$
where $\im(\delta_i)\subseteq \m F_{i-1}$ for every $i=1,...,p+1$. Since the entries of the matrixes representing the maps $d_i$ are all in $\m$, the same is true for all the entries of the matrixes representing the maps $d_i+\delta_i$. Thus, $F_{\delta}$ is part of a minimal free resolution of $R/J$. We conclude that $\beta_i^R(R/J)=\text{rank}F_i=\beta_i^R(R/I)$ for all $0\le i\le p$.

Let $n=\dim (R)$. For $N\gg 0$ we have that $\beta_i^R(R/J)=\beta_i^R(R/I)$ for all $0\le i\le n+1$. Denote by $\text{pd}_R(M)$ the projective dimension of a finitely generated $R$-module $M$. If $\text{pd}_R(R/I)=\infty$, then $\beta_{n+1}^R(R/J)=\beta_{n+1}^R(R/I)\neq 0$ and we conclude that $\text{pd}_R(R/J)=\infty$. Otherwise, if $\text{pd}_R(R/I)<\infty$, then $\text{pd}_R(R/I)=\max\{i:i\le n, \beta_i^R(R/I)\neq 0\}=\max\{i:i\le n,\beta_i^R(R/J)\neq 0\}=\text{pd}_R(R/J)$.
\end{proof}

\begin{remark} Notice that, unless $I$ has finite projective dimension, our proof does not show the existence of an integer $N$ such that $\beta_i^R(R/I) = \beta_i^R(R/J)$ for all $i \geq 0$. This is because the integer $N$ we pick in Theorem \ref{T-perturbatingcomplex} depends on the integer $p$ we fix at the beginning. Therefore, we ask the following question in the case of infinite projective dimension.
\end{remark}

\begin{question}
Let $I$ be an ideal of a Noetherian local ring $R$. Does there exist $N>0$ such that, for every ideal $J$ with $J\equiv I\bmod \m^N$ and with the same Hilbert function as $I$, all the Betti numbers of $R/I$ and $R/J$ coincide?
\end{question}

We believe that techniques similar to the ones used in this section can be applied to control other invariants under perturbations. In fact, inspired by the work of Pham Hung Quy and Van Duc Trung in \cite{quy2020small}, in upcoming work \cite{duarte2021localcohomology} we will use these methods to continue the study of local cohomology modules under small perturbations, provided the Hilbert function is preserved.


\section{Applications and examples}\label{SectionHilbert}

We start this section with some examples proving that the hypothesis on the results of the previous section cannot be relaxed. 

Proposition \ref{P-NumberGeneratorsInequality} shows that, given an ideal $I\subseteq R$, there exists $N>0$ such that $\beta_1^R(R/I)\le \beta_1^R(R/J)$ for all $J$ such that $I\equiv J\bmod \m^N$. The following example proves that this inequality cannot be extended to higher Betti numbers. Moreover, it shows that the assumption on the Hilbert function in Theorem \ref{T-perturbatingcomplex} cannot be dropped.

\begin{example}\label{ExampleConjecture}
Let $k$ be a field, and $R=k\llbracket x,y,z,w \rrbracket$. Consider the ideal $I=(f_1,f_2,f_3,f_4)$, where $f_1=x^2+z^5, f_2=xy+z^5, f_3=xz+w^5$ and $f_4=zw$, inside $R$. 
The minimal free resolution of $R/I$ over $R$ is
$$\begin{tikzcd}
0 \arrow[r] & R \arrow[r] & R^6 \arrow[r] & R^8 \arrow[r] & R^4 \arrow[r] & R \arrow[r] & R/I \arrow[r] & 0.
\end{tikzcd}$$ 
Now consider the ring $A=R\llbracket t_1,t_2,t_3,t_4 \rrbracket$, and let $J$ be the ideal generated by $f_1,f_2,f_3,f_4$ in $A$. Moreover, for every positive integer $N$, let $J_N$ be the ideal of $A$ generated by $f_1-t_1^N, f_2-t_2^N,f_3-t_3^N, f_4-t_4^N$, which is a regular sequence in $A$. We then have that $\beta_i^A(A/J_N) = \binom{4}{i}$  for all $N$ and all $0 \le i \leq 4$. On the other hand, $\beta_i^A(A/J) = \beta_i^R(R/I)$ for all $i$ and, in particular, $\beta_2^A(A/J) = 8$. Let $\m$ be the maximal ideal of $A$. This shows that, for every $N$, there exists $J_N \equiv J \bmod \m^N$ such that $\beta_2^A(A/J_N)<\beta_2^A(A/J)$.
\end{example}


Observe that in the previous example the ideals $J$ and $J_N$ have different heights. The next example shows that, even if the height is preserved, the Betti numbers of an ideal and its perturbations can be different.

\begin{example} \label{Example nCM}
Consider $R=k\llbracket x,y,z \rrbracket$ and $I=(x^2,y)$. Then $R/I$ is Cohen-Macaulay of dimension one. For every $N>0$ the ideal $I_N=(x^2,xy,y-z^N)$ is such that $I\equiv I_N\bmod \m^N$. $R/I_N$ has dimension one but is not Cohen-Macaulay. The minimal free resolution of $R/I_N$ is 
$$\begin{tikzcd}
0 \arrow[r] &R \arrow[r] & R^3 \arrow[r] & R^3 \arrow[r] & R \arrow[r] & 0,
\end{tikzcd}$$
so the Betti numbers of $R/I$ and $R/I_N$ are different. 
 
Observe that $R/I$ and $R/I_N$ have different Hilbert functions for every $N>0$. According to Proposition \ref{Cor-EqualStar}, for $N>\text{AR}((x,y,z),I\subseteq R)=2$ this can be seen by observing that $I^*\neq I_N^*$. Indeed, $xz^N\in I_N^*\smallsetminus I^*$. 
\end{example}



We now present some applications of our main result. There are several classes of rings where the assumptions of Theorem \ref{T-perturbatingcomplex} and Corollary \ref{T-SameBettiNumbers} are satisfied. An important result in this direction was recently proved by Ma, Pham Hung Quy and Smirnov. 

\begin{theorem}\label{Theorem-MaQuy}\emph{[\cite{ma2019filter}, Theorem 14]}
Let $(R,\m)$ be a Noetherian local ring and suppose $f_1,\ldots, f_r$ is a filter-regular sequence in $R$. Then there exists $N>0$ such that $R/(f_1,\ldots,f_r)$ and $R/(f_1+\epsilon_1,\ldots,f_r+\epsilon_r)$ have the same Hilbert function for every $\epsilon_1,\ldots,\epsilon_r\in\m^N$.
\end{theorem}

As a consequence, we obtain the following.

\begin{corollary}\label{T-SameBettiNumbersFilter}
Let $(R,\m)$ be a Noetherian local ring and let $p\in\N$. Suppose $f_1,\ldots, f_r$ is a filter-regular sequence in $R$. Then there exists $N>0$ such that, for every $\epsilon_1,\ldots,\epsilon_r\in\m^N$, we have
$$\beta_i^R(R/(f_1,\ldots,f_r))=\beta_i^R(R/(f_1+\epsilon_1,\ldots,f_r+\epsilon_r))$$ for every $0\le i\le p$.
\end{corollary}

We now prove two other results on the invariance of the Hilbert function under small perturbations. This provides further scenarios in which Corollary \ref{T-SameBettiNumbers} can be applied, that is, scenarios in which a finite number of Betti numbers is preserved under suitable perturbations.

We first need to recall some standard notions. Let $I$ be an ideal of $R$. The \emph{Hilbert series} $\text{HS}_{R/I}(t)=\sum_{n\ge 0}{\text{HF}_{R/I}(n)t^n}$ can be written as a quotient $f(t)/(1-t)^a$, with $f(t)\in\mathbb Z[t]$ and $a\in\N$. If $f(t)/(1-t)^a$ is written in irreducible form $h(t)/(1-t)^d$, i.e., with $h(1)\neq 0$, then $d=\dim(R/I)$. The polynomial $h(t)$ is called the $h$\emph{-polynomial} of $R/I$ and $e=h(1)$ is called the \emph{multiplicity} of $R/I$. 

There exists a polynomial $\text{HP}_{R/I}(X)\in\mathbb Q[X]$ of degree $d-1$, called the \emph{Hilbert polynomial} of $R/I$, such that $\text{HF}_{R/I}(n)=\text{HP}_{R/I}(n)$ for $n\gg 0$. Moreover, the coefficient of $X^{d-1}$ in $\text{HP}_{R/I}(X)$ coincides with $e/(d-1)!$. 
The \emph{regularity index} of $R/I$ is the smallest natural number $s$ for which $\text{HF}_{R/I}(n)=\text{HP}_{R/I}(n)$ for all $n\ge s$. If $l$ is the degree of the $h$-polynomial of $R/I$, then $s\le l-d+1$.

\vspace{0.2cm}
Our next theorem extends a result due to J. Elias \cite{elias1986analytic} in two directions. First, it applies to rings of any positive dimension. Secondly, it does not require the hypothesis that $R/I$ is Cohen-Macaulay.

\begin{theorem}\label{P-HighDepthPertutbation}
Let $(R,\m)$ be a Noetherian local ring and $I$ an ideal of $R$.
There exists $N>0$ such that, for every ideal $J \subseteq R$ satisfying the following conditions

\begin{itemize}
\item[(i)] $I\equiv J \bmod \m^N$,
\item[(ii)] $\dim(R/J)=\dim(R/I)=d$,
\item[(iii)] $R/J$ is Cohen-Macaulay,
\item[(iv)] $\emph{\text{depth }}\emph{\text{gr}}_{\m}(R/J)\ge d-1$,
\end{itemize}
one has that $I$ and $J$ have the same Hilbert function. 
\end{theorem}
\begin{proof}
Let $e$ and $e'$ be, respectively, the multiplicities of $R/I$ and $R/J$. Given $N>\text{AR}(\m,I\subseteq R)$, by Proposition \ref{Cor-EqualStar} we have that $\text{HF}_{R/I}(n)\ge \text{HF}_{R/J}(n)$ for every $n\in\N$, hence
$$e'=(d-1)!\lim_{n\to\infty}\frac{\text{HF}_{R/J}(n)}{n^{d-1}}\le (d-1)!\lim_{n\to\infty}\frac{\text{HF}_{R/I}(n)}{n^{d-1}}=e.$$ 
From a combination of Sally's machine (see \cite[Theorem 2.4]{rossi2011hilbert}) and \cite[Proposition 2.7]{rossi2010hilbert}, hypothesis (iii) and (iv) imply that the degree of the $h$-polynomial of $R/J$ is at most $e'-1 \le e-1$. Consequently, the regularity index of $R/J$ is $s'\le (e-1)-d+1=e-d.$  
Let $s$ be the regularity index of $R/I$ and choose $N\ge\max\{s,e-d\}+d$. 

From $I\equiv J \bmod \m^N$ we have $$\text{HF}_{R/J}(n)=\ell\left(\frac{J+\m^n}{J+\m^{n+1}}\right)=\ell\left(\frac{I+\m^n}{I+\m^{n+1}}\right)= \text{HF}_{R/I}(n)$$ for all $0\le n\le N-1$. Since $\max\{s,s'\}\le N-d$, it follows that the Hilbert polynomials of $R/I$ and $R/J$ coincide at the $d$ points $N-d,\ldots,N-1$. As both polynomials have degrees $d-1$, it follows that the two Hilbert polynomials must be the same. We conclude that the Hilbert functions of $R/I$ and $R/J$ are also the same.
\end{proof}

\begin{remark} Example \ref{Example nCM} shows that the assumption that $R/J$ is Cohen-Macaulay is indeed needed in Theorem \ref{P-HighDepthPertutbation} to conclude that $R/I$ and $R/J$ have the same Hilbert function.
\end{remark}

\begin{corollary}\label{Cor-EliasGen}
Let $(R,\m)$ be a Noetherian local ring, $I$ an ideal of $R$ and $p\in\N$. There exists $N>0$ such that, for every ideal $J$ satisfying the conditions of Theorem \ref{P-HighDepthPertutbation}, we have $\beta_i^R(R/I)=\beta_i^R(R/J)$ for every $0\le i\le p$.
\end{corollary}
\begin{proof}
Follows from Theorem \ref{P-HighDepthPertutbation} and Corollary \ref{T-SameBettiNumbers}.
\end{proof}

We recall that the \emph{embedding codimension} of $R/I$ is defined as $$h=\mu(\m/I)-d=\text{HF}_{R/I}(1)-d,$$ where $d$ is the dimension of $R/I$. It coincides with the coefficient of the degree one term in the $h$-polynomial of $R/I$.

In the following result we show that, if $R/I$ has almost minimal multiplicity, then the hypothesis that $\text{depth }\grm(R/J)\ge \dim(R/I)-1$ in Theorem can be removed. 


\begin{proposition}\label{P-MinMultPertutbation}
Let $(R,\m)$ be a Noetherian local ring, $I$ an ideal of $R$ and $p\in\N$. Suppose $e\le h+2$, where $e$ and $h$ are respectively the multiplicity and embedding codimension of $R/I$. Then there exists $N>0$ with the following property. If $J$ is an ideal such that $I \equiv J \bmod \m^N$, then $\beta_i^R(R/I)=\beta_i^R(R/J)$ for every $0\le i\le p$, provided $R/J$ is Cohen-Macaulay of the same dimension as $R/I$.
\end{proposition}
\begin{proof}
Let $d=\dim(R/I)$ and $e'$ be the multiplicity of $R/J$. Then, as in the proof of Theorem \ref{P-HighDepthPertutbation}, given $N\gg 0$ we have that $e'\le e$. Notice also that the embedding codimension of $R/J$ is $$\mu(\m/J)-d=\mu(\m/I)-d=h.$$ 
Using Abhyankar's inequality (see \cite[Theorem 4.1]{rossi2011hilbert}) we have 
$$ h+1\le e'\le e\le h+2,$$
thus implying $e'=h+1$ or $e'=h+2$. By \cite{sally1977associated} and \cite{rossi1996conjecture}, we have that $\grm(R/J)$ is Cohen-Macaulay in the first case and $\text{depth }\grm(R/J)\ge d-1$ in the second case. The result now follows by Corollary \ref{Cor-EliasGen}.
\end{proof}

\begin{remark}
Let $d=\dim(R/I)$ and let $e$ and $s$ be the multiplicity and regularity index of $R/I$, respectively. From the proofs of Theorem \ref{P-HighDepthPertutbation} and Proposition \ref{P-MinMultPertutbation}, we see that we can take $\max\{\text{AR}(\m,I\subseteq R)+1, \max\{s,e-d\}+d\}=\max\{\text{AR}(\m,I\subseteq R)+1, s+d,e\}$ as a lower bound for $N$ in both statements.
\end{remark}

\vspace{0.1cm}
\textbf{Acknowledgements:} I want to thank my advisors Maria Evelina Rossi and Alessandro De Stefani for their guidance and their helpful suggestions during this project. I am also grateful to the department of Mathematics of the University of Genova for supporting my PhD program. 

\bibliographystyle{plain}
\bibliography{citations}
\end{document}